\documentclass[10pt,oneside,reqno]{amsart}
\RequirePackage{fix-cm} 
\usepackage[english]{babel}
\usepackage{amssymb,amsmath,amsthm,graphicx,enumerate}
\usepackage[margin=1in,headsep=0.6in]{geometry}
\usepackage[dvipsnames]{xcolor}
\usepackage{lmodern}
\usepackage[T1]{fontenc}
\usepackage{hyperref}
\hypersetup{
 colorlinks,
 citecolor=Maroon,
 linkcolor=Maroon,
 urlcolor=Maroon}

\newtheorem{lemma}{Lemma}

\newtheorem{theorem}{Theorem}

\theoremstyle{remark}

\newcommand{\tmop}[1]{\ensuremath{\operatorname{#1}}}

\newcommand{\tmtextit}[1]{{\itshape{#1}}}

\newcommand{\tmtextsc}[1]{{\scshape{#1}}}

\newcommand{\RR}{\mathbb{R}}

\makeatletter
\def\@seccntformat#1{\hspace*{0mm}%
  \protect\textup{\protect\@secnumfont
    \ifnum\pdfstrcmp{subsection}{#1}=0 \bfseries\fi
    \csname the#1\endcsname
    \protect\@secnumpunct
  }%
}
\makeatother
 
\normalfont
      plus .8\fontdimen3\font
     minus .8\fontdimen4 

\begin{document}

\title[Local regularity of distributional solutions]{A short proof of local regularity\\of distributional solutions of \\Poisson's equation}

\author{Giovanni Di Fratta}
\address{Institute for Analysis and Scientific Computing, TU Wien, Wiedner
Hauptstrae 8-10, 1040 Wien, Austria}
\email{giovanni.difratta@asc.tuwien.ac.at}

\author{Alberto Fiorenza}
\address{Dipartimento di Architettura, Universita di Napoli, Via Monteoliveto,
3, I-80134 Napoli, Italy, and Istituto per le Applicazioni del Calcolo
``Mauro
Picone'', sezione di Napoli, Consiglio Nazionale delle Ricerche, via Pietro
Castellino,
111, I-80131 Napoli, Italy}
\email{fiorenza@unina.it}

\begin{abstract} We prove a local regularity result for distributional solutions of the Poisson's equation with $L^p$ data. We use a very short argument based on Weyl's lemma and Riesz-Fr\'echet representation theorem.
\end{abstract}

\date{April 6, 2019}

{\maketitle}

\section{Introduction}

Following the pleasant introduction on the regularity theory of elliptic equations in \cite{jost}, if $u\in C^3_0(\Omega)$, $\Omega$ being an open, bounded set in $\RR^n$, $n\ge 2$, then, using integrations by parts and Schwarz's theorem, and identifying the continuous, compactly supported functions with their corresponding elements in $L^2(\Omega)$,
\begin{align*}
\|D^2u\|_{L^2(\Omega)}^2 &=\sum_{i,j}^n\int_\Omega\ \partial_{ij}u\, \partial_{ij}u\, dx  =-\sum_{i,j}^n\int_\Omega \partial_{iji}u\,\partial_j u\,
dx \, =\, \sum_{i,j}^n\int_\Omega \partial_{ii}u\,\partial_{jj}u\,dx =\|\Delta u\|_{L^2(\Omega)}^2\, .
\end{align*}
This means that if $u\in C^3_0(\Omega)$ solves, in the classical sense, the Poisson's equation
\begin{equation}\label{poisson}
\Delta u=f\, ,
\end{equation}
then the $L^2$ norm of the datum $f$ controls the $L^2$ norm of \emph{all} second derivatives of $u$. This statement is a typical example of a result in the theory of elliptic regularity, whose main aim is to deduce this kind of results, but under weaker \emph{a priori} hypotheses on the regularity of the solution $u$.

\subsection{The notions of weak, very weak, and distributional solutions}
For a given $f\in L^2(\Omega)$, it is natural to study equation \eqref{poisson} in the $weak$ $sense$. This amounts to interpret \eqref{poisson} as equality between elements of the dual space $W^{-1,2}(\Omega)$ of the Sobolev space $W_0^{1,2}(\Omega)$; their images, when tested on every element $v\in W_0^{1,2}(\Omega)$, must coincide. If one looks for functions $u$ in $W_0^{1,2}(\Omega)$ which satisfy \eqref{poisson} in the weak sense (i.e., for \emph{weak solutions}), the requirement is that
\begin{equation}\label{weak}
-\int_\Omega \nabla u\cdot \nabla v\, dx=\int_\Omega fv\, dx\qquad \forall v\in W_0^{1,2}(\Omega)\, .
\end{equation}
Since $W_0^{1,2}(\Omega)$ is a Hilbert space, by the Riesz-Fr\'echet representation theorem (see, e.g., \cite[p.\,118]{fritzjohn}, \cite[Theorem 5.5 p.\,135]{brezis2010functional}), one gets the existence and uniqueness of the \emph{weak} solution.
Note, however, that the weak formulation \eqref{weak} relies on the apriori assumption that the solution $u$ has first derivatives with the same integrability property of the datum. For such solutions, one can prove the $W_{loc}^{2,2}(\Omega)$ regularity (see, e.g., \cite[Theorem 8.2.1]{jost}). 
 Also, we recall that under specific assumptions on the regularity of $\Omega$, one can get a better $global$ regularity for $u$, while under regularity assumptions on the datum, one can get a better \emph{local} regularity result for the solution (see, e.g., \cite[Theorem 8.2.2 and Corollary 8.2.1]{jost}). 
 
 By \eqref{weak} one gets the following equivalent equation (the equivalence with \eqref{weak} immediately follows from a standard density argument), where now the test functions $v$ are in $C_0^{\infty}(\Omega)$:
\begin{equation}\label{weakdistr}
-\int_\Omega \nabla u\cdot\nabla v\, dx=\int_\Omega fv\, dx\qquad \forall v\in C_0^{\infty}(\Omega)\, .
\end{equation}
When the problem is in this form, one can look for solutions of equation \eqref{poisson} in the space $W^{1,1}(\Omega)$, because the $L^2(\Omega)$ integrability of the gradient is not needed to give sense to the equation. Regularity results when the datum is in $L^p(\Omega)$, $1<p<\infty$, are classical, and rely upon the well known Calder\'on-Zygmund inequality from which one can get the $W_{loc}^{2,p}(\Omega)$ regularity (see, e.g., \cite[Theorem 9.2.2, p.\,248]{jost}, or \cite[Corollary 9.10, p.\,235]{GT}). We mention here also 
the method of difference quotients introduced by Nirenberg (see, e.g., \cite{douglis1955interior, nirenberg2011elliptic}, \cite[Step 1, p.\,121]{dacorogna2014introduction}, and \cite[Theorem 9.1.2 p.245]{jost}). 

Equation \eqref{weakdistr} is a special case of a class of linear equations which can be written in the form
\begin{equation}\label{general}
-{\rm div}(A\nabla u)=f
\end{equation}
for which it is known (\cite{serrin}) that even in the case $f\equiv 0$, when $A$ is a matrix function whose entries are locally $L^1$, there exist weak solutions, assumed a priori in $W_{loc}^{1,1}(\Omega)$, which are not in $W_{loc}^{1,2}(\Omega)$. As soon as one assumes that a solution is in $W_{loc}^{1,2}(\Omega)$, much local regularity can be gained by the celebrated De Giorgi's theorem (see e.g. \cite[Chap. 8]{GT} and references therein).

We recall, in passing, that for $\Omega\subset \RR^2$, it is possible to prove an existence and uniqueness theorem for weak solutions of \eqref{general} (and even for a nonlinear variant) in a space slightly larger than $W_0^{1,2}(\Omega)$, the so-called grand Sobolev space $W_0^{1,2)}(\Omega)$, when the datum is just in $L^1(\Omega)$ (see \cite[Theorem A]{FS}). For an excellent survey about solutions of a number of elliptic equations, called very weak because the solutions are assumed a priori in Sobolev spaces with exponents below the natural one, the reader is referred to \cite{IS}.

One can further weaken the notion of solution, and look for solutions of \eqref{poisson} 
in the space of regular distributions, that is, among elements of the dual space $ \mathcal{D}' (\Omega)$ of $C_c^{\infty} (\Omega)$ that can be identified with elements of $L^1_{loc}(\Omega)$.
In other words, their images, when computed in every element $\varphi\in C_c^{\infty} (\Omega)$, must coincide:
\begin{equation}\label{veryweakdistr}
\int_\Omega u\,\Delta \varphi\, dx=\int_\Omega f\varphi\, dx\qquad \forall \varphi\in C_0^{\infty}(\Omega)\, .
\end{equation}
Solutions in $L^1_{loc}(\Omega)$ of equation \eqref{veryweakdistr} are called {\it very weak solutions}. 

A condition on the datum $f$ ensuring existence and uniqueness in $L^1(\Omega)$ has been found in \cite[Lemma 1]{bcmr}, namely, if $f$ is in the weighted Lebesgue space where the weight is the distance function from the boundary of $\Omega$, there exists a unique solution $u\in L^1(\Omega)$ such that 
$$
\|u\|_{L^1(\Omega)}\le \|f \cdot {\rm dist}(x,\partial\Omega)\|_{L^1(\Omega)}\, .
$$
Differentiability results for very weak solutions are treated in a number of papers, see, e.g.,  \cite{DiazRakoJFA2009,rakoJFA2012} and references therein (see also \cite[Theorem 4.2]{FFR}). However, all such references gain regularity from data in weighted Lebesgue spaces, where the distance to the boundary is involved in the weight and the domain $\Omega$ has itself some regularity assumptions.

When the datum $f$ is identically zero, the masterpiece theorem of regularity for very weak solutions has been proved by Hermann Weyl in \cite[pp.\,415/6]{weyl1940method}. It dates back to 1940, well before the introduction of Sobolev spaces \cite{naumann2002remarks,lutzen2012prehistory}, and is nowadays referred to as Weyl's Lemma:
\begin{lemma}[H.~Weyl, 1940]\label{weyl}
Let $\Omega\subset \RR^n$ be an open set. Suppose that $u\in L^1_{loc}(\Omega)$ and
$$
\int_\Omega u\, \Delta\varphi dx=0\qquad \forall \varphi\in C_0^\infty(\Omega)\, .
$$
Then there exists a unique $\widetilde{u}\in C^\infty(\Omega)$ such that $\Delta\widetilde{u}=0$ in $\Omega$ and $\widetilde{u}=u$ a.e. in $\Omega$.
\end{lemma}
The proof given by Weyl in \cite{weyl1940method} is elementary and clever. Modern rephrasing of the proof  can be found in classical textbooks (see, e.g., \cite[Corollary 1.2.1]{jost},  \cite[Theorem 4.7]{dacorogna2014introduction}, \cite[Appendix, n.2]{simader2006}, \cite{talenti2016calcolo}). A beautiful note devoted entirely on this result and its development is the paper by Strook \cite{strook}, where Weyl's lemma is stated under the weaker assumption that $u\in \mathcal{D}' (\Omega)$. Indeed, one can go still further, and write \eqref{veryweakdistr} (in fact, \eqref{poisson}) in the form
\begin{equation}\label{veryweakdistr2}
\langle u,\Delta\varphi\rangle=\langle f,\varphi\rangle\qquad \forall \varphi\in C_0^{\infty}(\Omega)\, .
\end{equation}
Any solution of equation \eqref{veryweakdistr2}, is called a {\it distributional solution} of the Poisson's equation \eqref{poisson}.
The statement proved therein is the following (see also {\cite{schwartz1957theorie,hormander1983analysis,zuily2002elements}} for a more general result, valid for a broader class of differential operators). 

\begin{lemma}[Weyl's lemma in  $\mathcal{D}' (\Omega)$]\label{weyl2}
Let $\Omega\subset \RR^n$ be an open set. Suppose that $u\in \mathcal{D}' (\Omega)$ satisfies $\Delta u=f\in C^\infty(\Omega)$ in the sense of \eqref{veryweakdistr2}.
Then $u\in C^\infty(\Omega)$. 
\end{lemma}
The proof in \cite{strook} is very short and elegant. For our purposes, however, it is sufficient the particular case $f\equiv 0$, for which the proof in \cite{strook} further simplifies.

\subsection{Contributions of  present work}
In this note, we are interested in regularity results for \emph{distributional}  solutions of \eqref{poisson}, i.e.,  solutions satisfying \eqref{veryweakdistr2}. We prove a local regularity result for distributional solutions of the Poisson's equation with $L^p$ data. We use a concise argument based on Weyl's lemma and Riesz-Fr\'echet representation theorem. As a byproduct, we get the following classical result on \emph{very weak} solutions

\begin{theorem}\label{main}
If $f\in L^2_{loc}(\Omega)$, then any solution $u \in L^2_{loc} (\Omega)$ of $\Delta u=f$ $($i.e., satisfying \eqref{veryweakdistr}$)$ belongs to $W_{loc}^{2,2}(\Omega)$.
\end{theorem}

Theorem \ref{main}, known since 1965 (see \cite[Theorem 6.2 p.\,58]{agmon} for a more general result, proved for uniformly elliptic operators with Lipschitz continuous coefficients), is also quoted in the Brezis book~{\cite[Remark 25 p.\,306]{brezis2010functional}}, where it is claimed the delicateness of the proof of interior regularity of very
weak solutions, based on estimates for the difference quotient operator (see \cite[Def. 3.3 p.\,42]{agmon}). In \cite[Section 3 p.\,92]{ashton} the reader can find a modern proof, valid for a wide class of operators, which uses a precise estimate by H\"ormander in combination with a spectral representation for hypoelliptic operators. We quote also \cite[Theorem 1.3]{zhangbao}, where for general operators with locally Lipschitz continuous coefficients, in the case $f\equiv 0$, it is shown that very weak solutions in $L^1_{loc}(\Omega)$ are in fact in $W_{loc}^{2,p}(\Omega)$ for every $p\in [1,\infty)$; in \cite[Proposition 1.1]{zhangbao2}, the same authors, for general operators having locally Lipschitz continuous coefficients, in the case $f\in L^p_{loc}(\Omega)$, $1<p<\infty$, get that very weak solutions in $L^1_{loc}(\Omega)$ are in fact in $W_{loc}^{2,p}(\Omega)$.

For other results of regularity for very weak solutions of the Poisson's equation, see, e.g., \cite[Section 7.2 p.\,223]{mitrea} and \cite[Section 4.1 p.\,198]{necas}. In particular, we mention here that, following Hilbert, one can ask whether a solution, being a distribution, is analytic in the case where the right-hand side $f$ is analytic: the answer is positive for equation \eqref{general} when $A$ is analytic, see \cite{john2}.

\section{Regularity of very weak solutions of Poisson's equation in the $L^2$-setting}


The main ingredient is stated in the following
result which, remarkably, is essentially based on Weyl's lemma.

\begin{lemma}
  \label{prop3}Let $\Omega \subset \RR^n$ be an open set, and let $u \in
  \mathcal{D}' (\Omega)$. Then
  \begin{equation}
     \Delta u \in W^{-1,2}(\Omega) \; \Longrightarrow \; u \in
    W_{loc}^{1,2}(\Omega)\, .
  \end{equation}
\end{lemma}

\begin{proof}
  Since $\Delta u \in W^{-1,2} (\Omega)$, by Riesz representation theorem,
  there exists $v \in W_0^{1,2} (\Omega)$ such that $ \Delta v =  \Delta u$ in
  $\mathcal{D}' (\Omega)$. In particular, $ \Delta (u - v) = 0$ in
  $\mathcal{D}' (\Omega)$. By Weyl's lemma (Lemma \ref{weyl2} used with $f\equiv 0$), we know that $u - v \in
  C^{\infty} (\Omega)$. Hence $u = (u - v) + v \in C^{\infty} (\Omega) + W_0^{1,2}
  (\Omega) \subset W_{loc}^{1,2}(\Omega)$.
\end{proof}

We remark that solutions of Dirichlet problems by Hilbert spaces methods are a classic matter for $weak$ solutions, see, e.g., \cite[p.\,117]{fritzjohn}. Again, for weak solutions, we quote \cite[Lemma 2.1 p.48]{simader1996dirichlet}, where from the assumption of being locally in a Sobolev space, the authors get a better local regularity, still in Sobolev spaces.

\begin{theorem}
  Let $\Omega \subseteq \RR^n$ be an open set, and let $u \in \mathcal{D}'
  (\Omega)$. If $ \Delta u \in L^2_{loc} (\Omega)$, then $\nabla u \in
  W_{loc}^{1,2}(\Omega)$. If, in addition, $u \in L^2_{loc} (\Omega)$, then $u
  \in W_{loc}^{2,2} (\Omega)$.
\end{theorem}

\begin{proof}
  Due to the local character of the result, we can assume $\Delta u \in L^2(\Omega)$. Therefore, it is sufficient to note that if $ \Delta u = f$ with $f \in L^2 (\Omega)$
  then, for any distributional partial derivative of $u$, we have $ \Delta
  (\nabla u) = \nabla f$ with $\nabla f \in W^{-1,2} (\Omega)$. By
  the previous lemma, we get $\nabla u \in W_{loc}^{1,2}(\Omega)$. Thus, $u \in W_{loc}^{2,2} (\Omega)$ if we assume $u \in L^2_{loc} (\Omega)$.
\end{proof}

\section{Regularity of very weak solutions of Poisson's Equation in the $L^p$-setting}

We point out that the same argument shows that if $f \in L^p
(\Omega)$, $1<p<\infty$, then $u \in W^{2, p}_{\tmop{loc}} (\Omega)$. Precisely, the following result holds:
\begin{theorem}
  \label{prop3Lp}Let $\Omega \subseteq \RR^n$ be an open set, and let $u \in \mathcal{D}'
  (\Omega)$. If $ \Delta u \in L^p_{loc} (\Omega)$, then $\nabla u \in
  W_{loc}^{1,p}(\Omega)$. If, in addition, $u \in L^p_{loc} (\Omega)$, then $u
  \in W_{loc}^{2,p} (\Omega)$.
\end{theorem}

\begin{proof}
 Indeed
(see, e.g., \cite[pp.\,10-11]{simader1996dirichlet}), if $1 / p + 1 / q = 1$ and $F \in W^{-1, q'} (\Omega)$ then there exists a function $u_F \in W^{1, p}
(\Omega)$ such that
\begin{equation}
 - \int_{\Omega} \nabla u_F \cdot \nabla \varphi = \langle F, \varphi \rangle
  \label{eq:Rieszq}
\end{equation}
for every $\varphi \in W^{1, q}_0 (\Omega)$. Note that this can be considered
as the $q$-exponent version of the Riesz representation theorem. Now, {\eqref{eq:Rieszq}} implies that for any $F \in W^{-1, q'} (\Omega)$ there exists a distribution in $v_F \in W^{1, p} (\Omega)$
such that $ \Delta v_F = F$ in $\mathcal{D}' (\Omega)$. After that, assume
that $f \in L^p (\Omega)$ and $u \in \mathcal{D}'(\Omega)$ is a distributional solution of
\eqref{veryweakdistr2}. Then $\nabla u$ satisfies $ \Delta (\nabla u) = \nabla f$
with $\nabla f \in W^{-1, q'} (\Omega)$. Therefore, as in
Lemma~\ref{prop3}, $\nabla u \in W^{1, p}_{\tmop{loc}} (\Omega)$, and we
conclude.
\end{proof}

\section{Acknowledgments}
The first author acknowledges support from the Austrian Science Fund (FWF)
through the special research program {\emph{Taming complexity in partial
differential systems}} (Grant SFB F65).

\end{document}